\documentclass{AIMS}
\usepackage{amsmath}
  \usepackage{paralist}
  \usepackage{graphics} 
  \usepackage{epsfig} 
 \usepackage[colorlinks=true]{hyperref}
\usepackage{enumerate}

\usepackage{amssymb}
\usepackage{graphicx}
\usepackage{times}
\usepackage{color}
\usepackage{xspace,colortbl}
\DeclareGraphicsExtensions{.png,.pdf,.eps}
\hypersetup{urlcolor=blue, citecolor=red}

\def\currentvolume{X}
 \def\currentissue{X}
  \def\currentyear{200X}
   \def\currentmonth{XX}
    \def\ppages{X--XX}

\newtheorem{theorem}{Theorem}[section]
\newtheorem{corollary}[theorem]{Corollary}

\newtheorem{lemma}[theorem]{Lemma}
\newtheorem{proposition}[theorem]{Proposition}

\theoremstyle{definition}
\newtheorem{definition}[theorem]{Definition}

\newtheorem{example}[theorem]{Example}
\newcommand{\onetom}{1,\dots,m}
\newcommand{\oneton}{1,\dots,n}

\newcommand{\onetoK}{1,\dots,K}

\title[Consensus and synchronization in networks]
      {Consensus and synchronization in discrete-time networks of
multi-agents with stochastically switching topologies and time
delays}

\author[W. Lu, F. M. Atay and J. Jost]{}


 \email{wenlian@fudan.edu.cn}
 \email{atay@mis.mpg.de}
 \email{jost@mis.mpg.de}

\begin{document}
\maketitle

\centerline{\scshape Wenlian Lu }
\medskip
{\footnotesize
 \centerline{Center for Computational Systems Biology, School of Mathematical Sciences,}
   \centerline{Fudan University, Shanghai, China}
   \centerline{and Max Planck Institute for Mathematics in the Sciences,}
   \centerline{Inselstr. 22, 04103 Leipzig, Germany}
} 

\medskip

\centerline{\scshape Fatihcan M. Atay}
\medskip
{\footnotesize
 \centerline{Max Planck Institute for Mathematics in the Sciences,}
   \centerline{Inselstr. 22, 04103 Leipzig, Germany}
   }

\medskip

\centerline{\scshape J\"{u}rgen Jost}
\medskip
{\footnotesize
 \centerline{Max Planck Institute for Mathematics in the Sciences,}
   \centerline{Inselstr. 22, 04103 Leipzig, Germany}
\centerline{and Santa Fe Institute for the Sciences of Complexity}
\centerline{1399 Hyde Park Road, Santa Fe, NM 87501, USA}
   }

\bigskip

 \centerline{(Communicated by the associate editor name)}

\begin{abstract}
We analyze stability of consensus algorithms in networks of
multi-agents with time-varying topologies and delays. The topology
and delays are modeled as induced by an adapted process and are rather
general, including i.i.d.\ topology processes, asynchronous consensus
algorithms, and Markovian jumping switching. In case the self-links are instantaneous, we prove that the network
reaches consensus for all bounded delays if the graph corresponding to the
conditional expectation of the coupling matrix sum across a finite
time interval has a spanning tree almost surely.
Moreover, when self-links are also delayed
and when the delays satisfy certain integer patterns, we
observe and prove that the algorithm may not reach consensus but
instead synchronize at a periodic trajectory, whose period depends
on the delay pattern. We also give a brief discussion on the
dynamics in the absence of self-links.

\end{abstract}

\section{Introduction}
Consensus problems have been recognized as important in distribution
coordination of dynamic agent systems, which is widely applied in
distributed computing \cite{Lyn}, management science \cite{DeG},
flocking/swarming theory \cite{Vic}, distributed control \cite{Fax},
and sensor networks \cite{Olf}. In these applications, the
multi-agent systems need to agree on a common value for a certain
quantity of interest that depends on the states of the interests of
all agents or is a preassigned value.
The interaction rule for each agent specifying the information
communication between itself and its neighborhood is called the
\textit{consensus protocol/algorithm}. A related concept of
consensus, namely \textit{synchronization}, is considered as
``coherence of different processes'', and is a widely existing
phenomenon in physics and biology. Synchronization of interacting
systems has been one of the focal points in many research and
application fields \cite{Win,Kur,Pik}. For more details on consensus
and the relation between consensus and synchronization, the reader is
referred to the survey paper \cite{Olf1} and the references therein.

A basic idea to solve the consensus problem is updating the current
state of each agent by averaging the previous states of its
neighborhood and its own. The question then is whether or under
which circumstances the multi-agent system can reach consensus by
the proposed algorithm. In the past decade, the stability analysis
of consensus algorithms has  attracted much attention in control
theory and mathematics \cite{Olf1}. The
core purpose of stability analysis is not only to obtain the
algebraic conditions for consensus, but also to get the consensus
properties of the topology of the network. The basic discrete-time
consensus algorithm can be formulated as follows:
\begin{eqnarray}
x_{i}^{t+1}=x_{i}^{t}+\epsilon\sum_{j\in\mathcal
N_{i}}(x_{j}^{t}-x^{t}_{i}),~i=\onetom,\label{form1}
\end{eqnarray}
where $x^{t}_{i}\in\mathbb R$ denotes the state variable of the
agent $i$, $t$ is the discrete-time, $\mathcal N_{i}$ denotes the
neighborhood of the agent $i$, and $\epsilon$ is the coupling
strength. Define $\mathcal{L}=[l_{ij}]_{i,j=1}^{m}$ as the Laplacian of the
graph of the network in the manner that $l_{ij}=1$ if
$i\ne j$ and a link from $j$ to $i$ exists, $l_{ij}=0$ if
that $i\ne j$ and no link from $j$ to $i$ exists, and
$l_{ii}=-\sum_{j\ne i}l_{ij}$.  With $G=I-\epsilon \mathcal{L}$,
(\ref{form1}) can be rewritten as
\begin{eqnarray}
x^{t+1}=G x^{t},\label{form2}
\end{eqnarray}
where $x^{t}=[x^{t}_{1},\dots,x^{t}_{m}]^{\top}$.
If the diagonal elements in $G$ are nonnegative, i.e.,
$0 \le \epsilon\le 1/\max_{i}l_{ii}$, then $G$ is a stochastic matrix.
Eq. (\ref{form2}) is a general model of the synchronous consensus
algorithm on a network with fixed topology. The network can be a
directed graph, for example, the leader-follower structure
\cite{Mes}, and may have weights.

In many real-world applications, the connection structure may change in time,
for instance when the agents are moving in physical space.
One must then consider time-varying topologies under link failure
or creation. The asynchronous consensus algorithm also indicates
that the updating rule varies in time \cite{FAT}. Thus, the
consensus algorithm becomes
\begin{eqnarray}
x^{t+1}=G(t)x^{t},\label{form3}
\end{eqnarray}
where the time-varying coupling matrix $G(t)$ expresses to the
time-varying topology.
We associate $G(t)$ with a directed graph at time $t$
(see Sec.~2), in which $G_{ij}(t)>0$ implies that
there is a link from $j$ to $i$ at time $t$,
which may be a \textit{self-link} if $i=j$.
Note that the self links in $G$ arise from the presence of the
$x_i$ on the right hand side of (\ref{form1});
they do not necessarily mean that the physical network
of multi-agents have self-loops.

Furthermore, delays
occur inevitably due to limited information transmission speed.
The consensus algorithm with transmission delays can be described as
\begin{eqnarray}
x^{t+1}_{i}=\sum_{j=1}^{m}G_{ij}(t)x^{t-\tau_{ij}^{t}}_{j},\label{form4}
\end{eqnarray}
where $\tau_{ij}^{t}\in\mathbb N$, $i,j=\onetom$, denotes the
time-dependent delay from vertex $j$ to $i$.
A link from $j$ to $i$ is called
 \textit{instantaneous} if $\tau_{ij}^{t}=0$
 $\forall t$, and  \textit{delayed} otherwise.

In this paper, we study a general consensus problem
in networks with time-varying topologies and time delays described by
\begin{eqnarray}
x^{t+1}_{i}=\sum\limits_{j=1}^{m}G_{ij}(\sigma^{t})x^{t-\tau_{ij}
(\sigma^{t})}_{j},\quad i=\onetom, \label{multi_delay1}
\end{eqnarray}
as well as the more general form
\begin{eqnarray}
x^{t+1}_{i}=\sum\limits_{\tau=0}^{\tau_{M}}\sum\limits_{j=1}^{m}G^{\tau}_{ij}
(\sigma^{t})x^{t-\tau}_{j},
\quad i=\onetom.\label{delayed}
\end{eqnarray}
Note that (\ref{multi_delay1}) can be put into the form (\ref{delayed})
by partitioning the inter-links according to delays,
where $\tau_M$ is the maximum delay.
However, (\ref{delayed}) is more general, as it in principle
allows for multiple links with different delays between the same pair of vertices.
In particular, there may exist both instantaneous and delayed self-links,
which may naturally arise in a model like (\ref{form1}) where the term
$x_i$ appears both by itself as well as under the summation sign.
In reference to (\ref{delayed}), we talk about \emph{self-link(s)} when $G_{ii}^\tau\neq0$, which may be \emph{instantaneous} or \emph{delayed} depending on whether $\tau=0$ or $\tau>0$, respectively.
In equations (\ref{multi_delay1})--(\ref{delayed}), $\sigma^{t}$
denotes a stochastic process,
$G(\sigma^{t})=[G_{ij}(\sigma^{t})]_{i,j=1}^{n}=[\sum_{\tau=0}^{\tau_M}G_{ij}^{\tau}(\sigma^t)]_{i,j=1}^{n}$ is a stochastic
matrix, $\tau_{ij}(\sigma^{t})\in\mathbb N$ is the stochastically-varying transmission delay from agent $j$ to agent $i$. This model can
describe, for instance, communications between randomly moving
agents, where the current locations of the agents, and
hence  the links between them, are regarded as stochastic.
Furthermore, the delays are also stochastic since they arise due to
the distances between agents. In this paper, $\{\sigma^{t}\}$ is
assumed to be an adapted stochastic process.

\begin{definition}(\emph{Adapted process})
Let $\{A_{k}\}$ be a stochastic process defined on the basic
probability space $\{\Omega,\mathcal{F},P\}$, with the state space
$\Omega$, the $\sigma$-algebra $\mathcal F$, and the probability
$\mathbb P$. Let $\{\mathcal{F}^{k}\}$ be a filtration, i.e., a
sequence of nondecreasing sub-$\sigma$-algebras of $\mathcal{F}$. If
$A_{k}$ is measurable with respect to (w.r.t.) $\mathcal{F}^{k}$,
then the sequence $\{A_{k},\mathcal{F}^{k}\}$ is called an adapted
process.
\end{definition}

Via a standard transformation, any stochastic process can be
regarded as an adapted process. Let $\{\xi_{t}\}$ be a stochastic
process in probability spaces $\{\Omega^{t},\mathcal H^{t},\mathbb
P^{t}\}$. Define $\Omega=\prod_{t}\Omega^{t}$, $\mathcal F$ and
$\mathbb P$ are both induced by $\prod_{t}\mathcal H^{t}$ and
$\prod_{t}\mathbb P^{t}$, where $\prod$ stands for the Cartesian product. Let $\sigma^{t}=[\xi^{k}]_{k=1}^{t}$ and
$\mathcal F^{t}$ be the minimal $\sigma$-algebra induced by
$\prod_{k=1}^{t}\mathcal H^{t}$. Then $\mathcal
F^{t}$ is a filtration. Thus, it is clear that the notion of an
adapted process is rather general, and it contains i.i.d. processes,
Markov chains, and so on, as special cases.

\textbf{Related work.} Many recent papers address the stability
analysis of consensus in networks of multi-agents. However, the
model (\ref{multi_delay1}) with delays we have proposed above is
more general than the existing models in the literature. We first
mention some papers where models of the form (\ref{form3}) are
treated. A result from \cite{Mor} shows that (\ref{form3}) can reach
consensus uniformly if and only if there exists $T>0$ such that the
union graph across any $T$-length time period has a spanning tree.
Ref. \cite{Cao1} derived a similar condition for reaching a
consensus via an equivalent concept: {\em strongly rooted} graph.
Our previous papers \cite{Lu1,Lu2} studied synchronization of nonlinear
dynamical systems of networks with time-varying topologies by a
similar method. Ref. \cite{XW} has pointed out that under the
assumption that self-links always exist and are instantaneous
(i.e.~without delays), the condition presented in Ref.~\cite{Mor} also
guarantees consensus with arbitrary bounded multiple delays.
However, this criterion may not work when the time-varying topology
involves randomness, because for any $T>0$, it might occur with
positive probability that the union graph across some $T$-length
time period does not have a spanning tree for any $T$. Refs.~\cite{Hat,Wu,AA} studied the consensus in networks under the circumstance
that the processes $\{G(t)\}_{t\ge 0}$ are independently and
identically distributed (i.i.d.) and \cite{Zhang} also investigated
the stability of consensus of multi-agent systems with Markovian
switching topology with finite states. In these papers, consensus is considered
in the almost sure sense. Ref.~\cite{Fag} studied a particular
situation with packet drop communication. The most related
literature to the current paper is \cite{LLC}, where a general
stochastic process, an adapted process, was introduced to model the
switching topology, which generalized the existing works
including i.i.d. and Markovian jumping topologies as special cases.
The authors proved that, if the $\delta$-graph (see its definition in Sec.~2.2) corresponding to the conditional expectation of the coupling matrix sum across a finite time interval has a spanning tree almost surely,
then the system reaches consensus. However, none of those works
considered the stochastic delays but rather assumed that self-links
always exist. There are also many papers concerned with the
continuous-time consensus algorithm on networks of agents with
time-varying topologies or delays. See Ref.~\cite{Olf2} for a
framework and Ref. \cite{Olf1} for a survey, as well as Refs.
\cite{Mor1,Bli,XW1,Mich}, among others. Also, there are papers
concerned with nonlinear coupling functions \cite{DJ} and general
coordination \cite{LMA}.

\textbf{Statement of contributions.} In the following sections, we
study the stability of the consensus of the delayed system
(\ref{multi_delay1}), where $\sigma^{t}$ is an adapted process.
First, we consider the case that each agent contains an
instantaneous self-link. In this case, we show that the same
conditions enabling the consensus of algorithms without transmission delays, as mentioned
in Ref. \cite{LLC}, can also guarantee consensus for the case of
arbitrary bounded delays. Second, in case that delays also occur at
the self-links (for example, when it costs time for each agent to
process its own information), and only certain delay patterns can
occur, we show that
the algorithm does not necessarily reach
consensus but may synchronize to a periodic trajectory instead. As we
show, the period of the synchronized state depends on the possible
delay patterns. Finally, we briefly study the situation without
self-links, and present consensus conditions based on the graph
topology and the product of coupling matrices.

The basic tools we use are theorems about product of stochastic
matrices and the results from probability theory. Ref.~\cite{CS} has
proved a necessary and sufficient condition for the convergence of
infinite stochastic matrix products, which involves the concept of
scramblingness. Ref. \cite{Wolf} provided a means to get scrambling
matrices (defined in Sec. ~2.2) from products of finite stochastic indecomposable aperiodic (SIA) matrices and Ref.~\cite{XW}
showed that an SIA matrix can be guaranteed if the corresponding
graph has a spanning tree and one of the roots has a self-link. The
Borel-Cantelli lemma \cite{Durrett} indicates that if the
conditional probability of the occurrence of SIA matrices in a
product of stochastic matrices is always positive, then it occurs
infinitely often. These previous results give a bridge connecting
the properties of stochastic matrices, graph topologies, and
probability theory which we will call upon in the present paper.

The paper is organized as follows. Introductory notations,
definitions, and lemmas are given in Sec.~2. The dynamics of the
consensus algorithms in networks of multi-agents with switching
topologies and delays, which are modeled as adapted processes, are
studied in Sec.~3. Applications of the results are provided in
Sec.~4 to i.i.d. and Markovian jumping switching. Proofs of theorem are presented in Sec.~5. Conclusions are
drawn in Sec.~6.

\section{Preliminaries}
This paper is written in terms of stochastic process and algebraic graph
theory. For the reader's convenience, we present some necessary
notations, definitions and lemmas in this section. In what follows,
$\underline{N}$ denotes the integers from $1$ to $N$, i.e.,
$\underline{N}=\{1,\dots,N\}$. For a vector
$v=[v_{1},\dots,v_{n}]^{\top}\in\mathbb R^{n}$, $\|v\|$ denotes
some norm to be specified, for instance, the $L^1$ norm
$\|v\|_{1}=\sum_{i=1}^{n}|v_{i}|$. $\mathbb N$ denotes the set of positive
integers and $\mathbb Z$ denotes the integers. For two
integers $i$ and $j$, we denote by $\langle i\rangle_{j}$ the
quotient integer set $\{kj+i:~k\in\mathbb Z\}$. The greatest common
divisor of the integers $i_{1},\dots,i_{K}$ is denoted
$\mbox{gcd}(i_{1},\dots,i_{K})$. The product $\prod_{k=1}^{n}B_{k}$
of matrices denotes the left matrix product $B_n \times \cdots
\times B_1$. For a matrix $A$, $A_{ij}$ or $[A]_{ij}$ denotes the entry of $A$ on the $i$th row
and $j$th column. In a block matrix $B$, $B_{ij}$ or $[B]_{ij}$ can also stand for its $i,j$-th block. For two
matrices $A$, $B$ of the same dimension, $A\ge B$ means
$A_{ij}\ge B_{ij}$ for all $i$, $j$, and the relations
$A>B$, $A<B$, and $A\le B$ are defined similarly. $I_{m}$ denotes the identity
matrix of dimension $m$.

\subsection{Probability theory}
$\{\Omega,\mathcal{F},\mathbb P\}$ is our general notation for a
probability space, which may be different in different contexts. In this notation, $\Omega$ stands for the state space, $\mathcal{F}$ the Borel $\sigma$-algebra, and $\mathbb P\{\cdot\}$ the probability on $\Omega$. $\mathbb E_{\mathbb P}\{\cdot\}$ is the
expectation with respect to $P$ (sometimes $\mathbb E$ for
simplicity, if no ambiguity arises). For any $\sigma$-algebra
$\mathcal{G}\subseteq \mathcal{F}$, $\mathbb E\{\cdot|\mathcal{G}\}$
($\mathbb P\{\cdot|\mathcal{G}\}$) is the conditional expectation
(probability, respectively) with respect to $\mathcal{G}$. It should
be noted that both $\mathbb E\{\cdot|\mathcal{G}\}$ and $\mathbb
P\{\cdot|\mathcal{G}\}$ are actually random variables measurable
w.r.t. $\mathcal{G}$. The following lemma provides the
general statement of the principle of large numbers.

\begin{lemma}\label{lemSecondBorelCantelli} \cite{Durrett} (The Second
Borel-Cantelli Lemma) Let $\mathcal{F}_{n}$, $n\ge 0$ be a
filtration with $\mathcal{F}_{0}=\{\emptyset,\Omega\}$ and $C_{n}$,
$n\ge 1$ a sequence of events with $C_{n}\in\mathcal{F}_{n}$. Then
\begin{eqnarray*}
\{C_{n}~~\textnormal{ infinitely often
}\}=\big\{\sum_{n=1}^{+\infty}P\{C_{n}|\mathcal{F}_{n-1}\}=+\infty\big\}
\end{eqnarray*}
with a probability $1$, where "infinitely often" means that an
infinite number of $\{C_{n}\}_{n=1}^{\infty}$ occur.
\end{lemma}

\subsection{Stochastic matrices and graphs}

An $m\times m$ matrix $A=[a_{ij}]_{i,j=1}^{m}$ is said to be a {\it
stochastic matrix} if $a_{ij}\ge 0$ for all $i,j=\onetom$ and
$\sum_{j=1}^{m}a_{ij}=1$ for all $i=\onetom$. A matrix $A\in\mathbb
R^{m,m}$ is said to be {\it SIA} if $A$ is stochastic,
indecomposable, and aperiodic, i.e., $\lim_{n\to\infty}A^{n}$ converges to a matrix with identical rows.
The {\it Hajnal diameter} is introduced
in Ref.~\cite{Haj1,Haj2} to describe the compression rate of a
stochastic matrix. For a matrix $A$ with row vectors
$a_{1},\dots,a_{m}$ and a vector norm $\|\cdot\|$ in
$\mathbb{R}^{m}$, the Hajnal diameter of $A$ is defined by $ {\rm
diam}(A)=\max\limits_{i,j}\|a_{i}-a_{j}\|$. The {\it scramblingness}
$\eta$ of a stochastic matrix $A$ is defined as
\begin{eqnarray}
\eta(A)=\min\limits_{i,j}\|a_{i}\wedge a_{j}\|_{1},\label{scramble}
\end{eqnarray}
where $ a_{i}\wedge
a_{j}=[\min(a_{i1},a_{j1}),\dots,\min(a_{im},a_{jm})]$. The
stochastic matrix $A$ is said to be {\it scrambling} if $\eta(A)>0$.
The Hajnal inequality estimates the Hajnal diameter of the product
of stochastic matrices. For two stochastic matrices $A$ and $B$ of
the same order, the inequality
\begin{eqnarray}
{\rm diam}(AB)\le(1-\eta(A)){\rm diam}(B)\label{Hajnal_inequality}
\end{eqnarray}
holds for any matrix norm \cite{Shen}. It can be seen from (\ref{Hajnal_inequality}) that the diameter of the  product $AB$  is strictly less than that of $B$ if $A$ is scrambling.

The link between stochastic matrices and graphs is
an essential feature of this paper.
A stochastic (or simply nonnegative) matrix $A=[a_{ij}]_{i,j=1}^{m}\in\mathbb R^{m,m}$
defines a graph $\mathcal G=\{\mathcal V,\mathcal E\}$, where
$\mathcal V=\{\onetom\}$ denotes the {\it vertex set} with $m$
vertices and $\mathcal E$ denotes the {\it link set} where there
exists a directed link from vertex $j$ to $i$, i.e., $e(i,j)$
exists, if and only if $a_{ij}>0$. We denote this graph
corresponding to the stochastic matrix $A$ by $\mathcal G(A)$. For a
directed link $e(i,j)$, we say that $j$ is the {\it start} of the
link and $i$ is the {\it end} of the link. The vertex $i$ is said to
be self-linked if $e(i,i)$ exists, i.e., $a_{ii}>0$. $\mathcal G$ is said to be a {\it
bigraph} if the existences of $e(i,j)$ and $e(j,i)$ are equivalent.
Otherwise, $\mathcal G$ is said to a {\it digraph}. An $L$-length
{\it path} in the graph denotes a vertex sequence
$(v_{i})_{i=1}^{L}$ satisfying that the link $e(v_{i+1},v_{i})$
exists for all $i=1,\dots,L-1$. The vertex $i$ can {\it access} the
vertex $j$, or equivalently, the vertex $j$ is {\it accessible} from
the vertex $i$, if there exists a path from the vertex $i$ to $j$.
The graph $\mathcal G$ has a {\it spanning tree} if there exists a
vertex $i$ which can access all other vertices, and
the set of vertices that can access all other vertices is named the
{\it root set}. The graph $\mathcal G$ is said to be {\it strongly
connected} if each vertex is a root. We refer interested readers to the book \cite{God} for more details.
Due to the relationship between nonnegative matrices and graphs, we can call on the properties of nonnegative matrices, or equivalently, those of their corresponding graphs. For example, the indecomposability of a nonnegative matrix $A$ is equivalent to that $\mathcal G(A)$ has a spanning tree, and the aperiodicity of a graph is associated with the aperiodicity of its corresponding matrix \cite{Horn}.
We say that $\mathcal G$ is {\it
scrambling} if for each pair of vertices $i\ne j$, there exists a
vertex $k$ such that both $e(i,k)$ and $e(j,k)$ exist, which can be
seen to be equivalent to the definition of scramblingness for stochastic matrices. For two matrices
$A=[a_{ij}]_{i,j=1}^{n},B=[b_{ij}]_{i,j=1}^{n}\in\mathbb R^{n,n}$,
 we say $A$ is an {\it analog} of $B$ and write $A\approx B$,
in case that $a_{ij}\ne 0$ if and only if $b_{ij}\ne 0$,  $\forall i,j=\oneton$, that is,
when their corresponding graphs are identical.

Furthermore, for a nonnegative matrix $A$ and a given $\delta>0$,
the {\em $\delta$-matrix} of $A$, denoted by $A^{\delta}$, is
defined as
\begin{eqnarray*}
[A^{\delta}]_{ij}=\left\{
\begin{array}{ll}
\delta,& \mbox{if }A_{ij}\ge \delta;\\
0,&  \mbox{if }A_{ij}<\delta.
\end{array}\right.
\end{eqnarray*}
The {\em$\delta$-graph} of $A$ is the directed graph corresponding
to the $\delta$-matrix of $A$. We denote by $\mathcal
N_{i}^{\delta}$ the neighborhood set of the vertex $v_{i}$ in
the $\delta$-graph: $\mathcal{N}_{i}^{\delta}=\{v_{j}: A_{ij}\ge\delta\}$.

\subsection{Convergence of products of stochastic matrices}
Here, we provide the definition of consensus and synchronization of
the system (\ref{multi_delay1}). Suppose the delays are bounded, namely,
$\tau_{ij}(\sigma^{k})\le\tau_{M}$ for all $i,j=\onetom$ and
$\sigma^{k}\in\Omega$.

\begin{definition}
The multi-agent system is said to {\it reach consensus} via the
algorithm (\ref{multi_delay1}) if for any essentially
bounded random initial data $x^{0}_{\tau}\in\mathbb R^{m}$,
$\tau=0,1,\dots,\tau_{M}$,
(that is, $x^{0}_{\tau}$ is bounded with probability one),
and almost every sequence
$\{\sigma^{t}\}$, there exists a number $\alpha\in\mathbb R$ such
that $ \lim\limits_{t\rightarrow\infty}x^{t}=\alpha {\mathbf 1}$
with $\mathbf 1=[1,1,\dots,1]^{\top}$.
The multi-agent system is
said to {\it synchronize} via the algorithm (\ref{multi_delay1}) if
for any initial essentially bounded random $x^{0}\in\mathbb R^{m}$ and almost every sequence
$\{\sigma^{t}\}$ , $\lim_{t\rightarrow\infty}|x_{i}(t)-x_{j}(t)|=0,~i,j=\onetom$. In
particular, if for any initial essentially
bounded random $x^{0}_{\tau}\in\mathbb R^{m}$, $\tau=0,1,\dots,\tau_{M}$, and almost
every sequence, there exists a $P$-periodic trajectory $s(t)$ ($P$
independent of the initial values and the sequence) such that
$\lim_{t\to\infty}|x_{i}(t)-s(t)|=0$ holds for all $i=\onetom$, then
the multi-agent system is said to {\it synchronize to a $P$-periodic
trajectory} via the algorithm (\ref{multi_delay1}).
\end{definition}

In general, consensus can be regarded as a special case of
synchronization, where the multi-agent system synchronizes at an
equilibrium. As shown in Ref.~\cite{CS}, in the absence of delays,
consensus and synchronization are equivalent w.r.t. the
product of infinite stochastic matrices; that is, whenever a system
synchronizes, it also reaches consensus. However, we will show in
the following sections that, under transmission delays, consensus
and synchronization of the algorithm (\ref{multi_delay1}) are not
equivalent. Thus, a system can synchronize without necessarily
reaching consensus.

Consider the model where the topologies are induced by a stochastic process:
\begin{eqnarray}
x^{t+1}_{i}=\sum\limits_{j=1}^{m}G_{ij}(\xi^{t})x^{t}_{j},~i=\onetom,
\label{general}
\end{eqnarray}
where $\{\xi^{t}\}_{t\in\mathbb N}$ is a stochastic process with a
probability distribution of the sequence $\mathbb P$.
The results of this paper are based on the following lemma, which is
a consequence of  Theorem 2 in Ref.~\cite{CS}.

\begin{lemma}  \label{basic_lemma}
Let $\eta(\cdot)$ denote the scramblingness,
as defined in (\ref{scramble}). The multi-agent system via the algorithm
(\ref{general}) reaches consensus if and only if for $\mathbb
P$-almost every sequence there exist infinitely many disjoint integer
intervals $I_{i}=[a_{i},b_{i}]$ such that
\begin{eqnarray*}
\sum\limits_{i=1}^{\infty}\eta\bigg(\prod\limits_{k=a_{i}}^{b_{i}}G(\xi^{k})
\bigg)=\infty.
\end{eqnarray*}
\end{lemma}

As a trivial extension to a set of SIA matrices, we have the next
lemma on how to obtain scramblingness.

\begin{lemma}\cite{Wolf}\label{SIA}
Let $\Theta\subset\mathbb R^{m,m}$ be a set of SIA matrices. There
exists an integer $N$ such that any $n$-length matrix sequence
with $n>N$ picked from $\Theta$: $G^{1},G^{2},\dots,G^{n}$
satisfies
\begin{eqnarray*}
\eta\bigg(\prod\limits_{k=1}^{n}G^{k}\bigg)>0.
\end{eqnarray*}
\end{lemma}

The following result provides a relation between SIA matrices
and spanning trees.

\begin{lemma} (Lemma 1 in Ref. \cite{XW})\label{spanning}
If the graph corresponding to a stochastic matrix $A$ has a spanning
tree and a self-link at one of its root vertices, then $A$ is SIA.
\end{lemma}

\section{Main results}
We first consider the multi-agent network without transmission delays:
\begin{eqnarray}
x^{t+1}_{i}=\sum\limits_{j=1}^{n}G_{ij}(\sigma^{t})x^{t}_{j},~i=\onetom.
\label{multi_undelay}
\end{eqnarray}
The following theorem is the main tool for the proofs of the main
results and it can be regarded as a realization of Lemma
\ref{basic_lemma} and an extension from Ref. \cite{LLC} without
assuming self-links.
\begin{theorem}\label{basic_thm}
For the system (\ref{multi_undelay}), if there exist $L\in\mathbb N$
and $\delta>0$ such that  the $\delta$-graph of the matrix
product
\begin{eqnarray}
\mathbb E\bigg\{\prod_{k=n+1}^{n+L}G(\sigma^{k})|\mathcal
F^{n}\bigg\}
\end{eqnarray}
has a spanning tree and is aperiodic for all $n\in\mathbb N$ almost surely, then the multi-agent
system reaches a consensus.
\end{theorem}
The proof is given in Sec.~5.1. The main result of \cite{LLC} can
be regarded as a consequence of Theorem \ref{basic_thm},
where each node in the graph was assumed to have a self-link. In the following, we
first study the multi-agent systems with transmission delays such
that each agent is linked to itself without delay and then
investigate the general situation where delays may occur also on the
self-links. Finally, we give a brief discussion on the consensus
algorithms without self-links. All proofs in this section are placed
in Sec.~5.

\subsection{Consensus and synchronization with transmission delays}
Consider the consensus algorithm (\ref{delayed}), which we rewrite in
matrix form as
\begin{eqnarray}
x^{t+1}=\sum\limits_{\tau=0}^{\tau_{M}}G^{\tau}(\sigma^{t})x^{t-\tau},
\label{delayed_matrix}
\end{eqnarray}
where
$G(\sigma^{t})=[G_{ij}^{\tau}(\sigma^t)]_{i,j=1}^{n}$.
We assume the following for the matrices
$G^{\tau}(\cdot)$.

$\mathbf A$: Each $G^{\tau}(\sigma^{t})$,
$\tau\in\underline{\tau_{M}}$, is a measurable map from $\Omega$ to
the set of nonnegative matrices with respect to $\mathcal F^{t}$.

Letting
$y^{t}=[{x^{t}}^{\top},{x^{t-1}}^{\top},\dots,{x^{t-\tau_{M}}}^{\top}]^{\top}
\in\mathbb
R^{m\times(\tau_{M}+1)}$, we can write (\ref{delayed_matrix}) as
\begin{eqnarray}
y^{t+1}=B(\sigma^{t})y^{t},\label{integ}
\end{eqnarray}
where $B(\sigma^{t})\in\mathbb R^{(\tau_{M}+1)\times
m,(\tau_{M}+1)\times m}$ has the form
\begin{eqnarray*}
B(\sigma^{t})=\left[\begin{array}{lllll}
G^{0}(\sigma^{t})&G^{1}(\sigma^{t})&\cdots&G^{\tau_{M}-1}(\sigma^{t})
&G^{\tau_{M}}\\
I_{m}&0&\cdots&0&0\\
0&I_{m}&\cdots&0&0\\
\vdots&\vdots&\ddots&\vdots&\vdots\\
0&0&\cdots&I_{m}&0\end{array}\right].
\end{eqnarray*}
Thus, the consensus of (\ref{delayed}) is equivalent to that of
(\ref{integ}). As a default labeling, let us consider the corresponding graph $\mathcal G(B(\sigma^{t}))$,
which has $(\tau_{M}+1)m$ vertices, which we denote by
$\{v_{i,j},~i\in\underline{\tau_{M}+1},~j\in\underline{m}\}$, where
$v_{i,j}$ corresponds to the ($(i-1)m+j$)th row (or column) of the matrix
$B(\sigma^{t})$.

\begin{theorem}\label{thm2}
Assume the conditions  $\mathbf A$, and suppose there
exist $\mu>0$, $L\in\mathbb N$, and $\delta>0$ such that $G^{0}(\sigma)
>\mu I_{m}$ for all
$\sigma\in\Omega$ and the $\delta$-graph of $\mathbb
E\{\sum_{k=n+1}^{n+L}G(\sigma^{k})|\mathcal F^{n}\}$ has a spanning
tree for all $n\in\mathbb N$ almost surely. Then the delayed
multi-agent system (\ref{delayed}) reaches consensus.
\end{theorem}

The proof is given in Sec~5.2. In the case that the topological switching is deterministic, a similar result is obtained in the literature \cite{Mor1,XW}.

\begin{example}
We give a simple example to illustrate Theorem \ref{thm2}. Consider
a delayed multi-agent system on a network with $2$ vertices and the
maximum delay is $1$. The system can be written as
\begin{eqnarray*}
x^{t+1}=G^{0}(\sigma^{t})x^{t}+G^{1}(\sigma^{t})x^{t-1},
\end{eqnarray*}
which can further be put into a form without delays
$y^{t+1}=B(\sigma^{t})y^{t}$ with
\begin{eqnarray*}
B(\sigma^{t})=\left(\begin{array}{ll}
G^{0}(\sigma^{t})&G^{1}(\sigma^{t})\\
I_{m}&0\end{array}\right).
\end{eqnarray*} Let us consider the product of two matrices
$B^{1}$ and $B^{2}$:
\begin{eqnarray*}
B^{1}=\left(\begin{array}{llll}1&0&0&0\\
0&1&0&0\\
1&0&0&0\\
0&1&0&0\end{array}\right),\quad B^{2}=\left(\begin{array}{llll}1/2&0&0&1/2\\
0&1&0&0\\
1&0&0&0\\
0&1&0&0\end{array}\right).
\end{eqnarray*}
In the absence of delays, they correspond to
$G_{1}=\left(\begin{array}{ll}1&0\\0&1\end{array}\right)$ and
$G_{2}=\left(\begin{array}{ll}1/2&1/2\\0&1\end{array}\right)$. One
can see that the union of the graphs $\mathcal G(G_{1})$ and
$\mathcal G(G_{2})$ has spanning trees and self-connections. Then
the proof of Theorem \ref{thm2} says that for some integer $L$, the
product of $L$ successive matrices corresponds to a graph which has a
spanning tree and a self-link on the root node. For example, we
consider the following matrix product:
\begin{eqnarray*}
B^{1}B^{2}=\left(\begin{array}{llll}1/2&0&0&1/2\\
0&1&0&0\\
1/2&0&0&1/2\\
0&1&0&0\end{array}\right).
\end{eqnarray*}
The corresponding graph has four vertices, which we label as $v_{1,1}$, $v_{1,2}$, $v_{2,1}$, and $v_{2,2}$
following the scheme defined below Eq.~(\ref{integ}). From Figure~\ref{fig0},
it can be seen that the graph corresponding to $B^{1}B^{2}$ has
spanning trees with $v_{1,2}$ being the root vertex which has a
self-link.
So, by Theorem \ref{thm2}, the system reaches consensus.
\begin{figure}[htp]
\begin{center}
\includegraphics[width=2.4in]{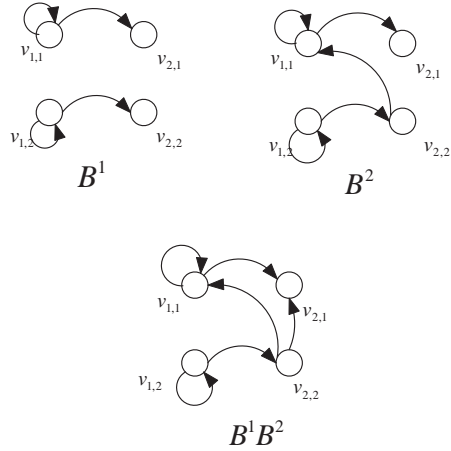}
\caption{The graphs
corresponding to the matrices$B^{1}$, $B^{2}$, and the matrix product $B^{1}B^{2}$, respectively.}\label{fig0}
\end{center}
\end{figure}
\end{example}

In some cases delays occur at self-links, for example, when it takes
time for each agent to process its own information. Suppose that the
self-linking delay for each vertex is identical, that is,
$\tau_{ii}=\tau_{0}>0$. We classify each integer $t$ in the
discrete-time set $\mathbb N$ (or the integer set $\mathbb Z$) via
$\mod(t+1,\tau_{0}+1)$ as the quotient group of $(\mathbb
Z+1)/(\tau_{0}+1)$. As a default set-up, we denote $\langle
i\rangle_{\tau_{0}+1}$ by $\langle i\rangle$. Let
$\hat{G}^{i}(\cdot)=\sum_{j\in \langle i\rangle}G^{j}(\cdot)$. For a
simplified statement of the result, we provide the following
condition \textbf{B}:

{\em \begin{enumerate}

\item[\textbf{B}.1] There exist an integer $\tau_{0}>0$ and a
number $\mu>0$ such that $G^{\tau_{0}}(\sigma_{1})>\mu I_{m}$ for
all $\sigma_{1}\in\Omega$;

\item[\textbf{B}.2] There exist
$\tau_{1},\dots,\tau_{K}$ excluding the integers in $\langle 0\rangle$ with
$\mbox{gcd}(\tau_{0}+1, \tau_{1}+1,\dots,\tau_{K}+1)=P>1$ such that
$\hat{G}^{j}(\sigma_{1})=0$ for all
$j\notin\{\tau_{1},\dots,\tau_{K}\}$ and all $\sigma_{1}\in\Omega$
and the $\delta$-matrix of $\mathbb
E\{\hat{G}^{\tau_{k}}(\sigma^{n+1}) |\mathcal F^{n}\}$ is nonzero
for all $n\in\mathbb N$ and $k=\onetoK$ almost surely.
\end{enumerate}}

\begin{theorem}\label{thm3} Assume that the conditions $\mathbf
A$ and $\mathbf B$ hold, and suppose there exist $L\in\mathbb N$ and $\delta>0$
such that the $\delta$-graph of $\mathbb
E\{\sum_{k=n+1}^{n+L}\hat{G}^{0}(\sigma^{k}) |\mathcal F^{n}\}$ is
strongly connected for all $n\in\mathbb N$ almost surely. Then the
system (\ref{delayed}) synchronizes to a $P$-periodic trajectory. In
particular, if $P=1$, then (\ref{delayed})
reaches consensus.
\end{theorem}

The proof is given in Sec.~5.3. From this theorem, one can
see that under self-linking delays, consensus is not equivalent to
synchronization. In fact, the delays that occur on self-links are essential for the failure to reach consensus.

\begin{example}
Theorem \ref{thm3} demands that the $\delta$-graph corresponding to
the matrix \\
$\mathbb E\{\sum_{t=n+1}^{n+L}\hat{G}^{0}(\sigma^{t})|\mathcal
F^{n}\}$ is strongly connected. This is stronger than the condition
in Theorem \ref{thm2}, which demands that the corresponding graph
has a spanning tree. We give an example to show that the strong
connectivity is necessary for the reasoning in the proof. Consider a
delayed multi-agent system on a network with two vertices and a
maximum delay of 3. Consider the form (\ref{integ}) and the matrix
$B(\cdot)$. Suppose that the state space only contains one state
$\sigma_{1}$ as follows:
\begin{eqnarray*}
B(\sigma_{1})=\left[\begin{array}{cccccccc} 0&0&1/3&0&0&1/3&0&1/3\\
0&0&0&1&0&0&0&0\\
1&0&0&0&0&0&0&0\\
0&1&0&0&0&0&0&0\\
0&0&1&0&0&0&0&0\\
0&0&0&1&0&0&0&0\\
0&0&0&0&1&0&0&0\\
0&0&0&0&0&1&0&0\end{array}\right].
\end{eqnarray*}
Here, $\tau_{0}=1$. It is clear that the subgraph corresponding to
each $\hat{G}^{0}_{1,2}$ has spanning trees but is not strongly
connected, and that there is a link between the subgraphs
corresponding to $\langle 1 \rangle$ and $\langle 0 \rangle$. For
the word $\sigma_{1}\sigma_{1}\cdots\sigma_{1}\sigma_{1}$, direct
calculations show that the corresponding matrix product is an analog of the following matrix if the length of the word is sufficiently
long:
\begin{eqnarray}
\left[\begin{array}{cccccccc} 1&1&0&1&0&1&0&0\\
0&1&0&0&0&0&0&0\\
0&1&1&1&0&1&0&1\\
0&0&0&1&0&0&0&0\\
1&1&0&1&0&1&0&0\\
0&1&0&0&0&0&0&0\\
0&1&1&1&0&1&0&1\\
0&0&0&1&0&0&0&0\end{array}\right]\label{prod1}
\end{eqnarray}
The corresponding graph is shown in Figure~\ref{fig1}, using the labeling scheme for the vertices as defined below Eq.~(\ref{integ}).
One can see that it does not have a spanning tree since the vertices
$v_{1,2}$ and $v_{2,2}$ do not have incoming links other than
self-links.
In fact, the set of eigenvalues of the matrix $B(\sigma_{1})$ contains
$1$ and $-1$, which implies that  (\ref{delayed_matrix}) with
$B(\sigma^{t})$ can not reach consensus even though the condition in
Theorem \ref{thm2} is satisfied.

\begin{figure}[]
\begin{center}
\includegraphics[width=2.4in]{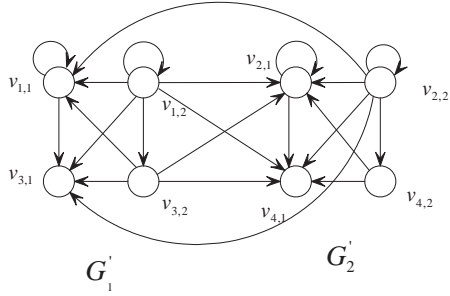}
\caption{The graph
corresponding to the matrix product (\ref{prod1}).}\label{fig1}
\end{center}
\end{figure}

\end{example}

\subsection{Consensus and synchronization without self-links}
So far the stability result is based on the assumption that each
agent takes its own state into considerations when updating. In
other words, the coupling matrix has positive diagonals (possibly
with delays). There also exist consensus algorithms that are
realized by updating each agent's state via averaging its neighbor's
states and possibly {\it excluding} its own \cite{FAT}. In
\cite{DeG}, it is shown that consensus can be reached in a static
network if  each agent can communicate with others by a directed
graph and the coupling graph is aperiodic, which can be proved by
nonnegative matrix theory \cite{Horn}. In the following, we briefly
discuss the general consensus algorithms in networks of
stochastically switching topologies that do not necessarily have
self-links for all vertices.

When transmission delays occur, the general algorithm
(\ref{delayed}) can be regarded as increasing dimensions as in
(\ref{integ}). Thus, one can similarly associate  (\ref{integ}) with
a new graph on $m\times(\tau_{M}+1)$ vertices
$\{v_{ij}:~i\in\underline{\tau_{M}+1},j\in\underline{m}\}$, denoted
by $\mathcal G^{'}(\cdot)$, where $B(\cdot)$ denotes the link set of
$\mathcal G^{'}(\cdot)$, by which $v_{ij}$ corresponds to the
$(i-1)\times(\tau_{M}+1)+j$ column and row of $B$.
$\hat{B}_{p}(\sigma_{1})$ as the matrix corresponding the vertices
$\{v_{ij}:~i\in\langle p\rangle,~j\in\underline{m}\}$.
Based on theorem \ref{basic_thm}, we
have the following results, which can be proved similarly to
Theorems \ref{thm2} and \ref{thm3}.

\begin{proposition}\label{prop1}
Assume {$\mathbf A$} holds, and suppose there exist  $L\in\mathbb
N$ and $\delta>0$ such that the $\delta$-graph of $\mathbb
E\{\prod_{k=u+1}^{u+L}B(\sigma^{k})|\mathcal F^{u}\}$ has a
spanning tree and self-link at one root vertex for all $n\in\mathbb
N$ almost surely. Then the algorithm (\ref{multi_undelay}) reaches
consensus.
\end{proposition}

In fact, under the stated conditions, each product $\mathbb
E\{\prod_{k=u+1}^{u+L}B(\sigma^{k})|\mathcal F^{u}\}$ is SIA almost surely; so,
this proposition is a direct consequence of Theorem \ref{basic_thm}.

In the possible absence of self-links, the following is a
consequence of Proposition \ref{prop1}.

\begin{proposition}\label{prop2}
Assume ${\mathbf A}$ and $\mathbf{B}.2$ hold ($\mathbf{B}.1$ need not hold). Suppose there exist $L\in\mathbb N$
and $\delta>0$ such that the $\delta$-graph of $\mathbb
E\{\prod_{k=n+1}^{n+L}\hat{B}_{p}(\sigma^{k})|\mathcal F^{n}\}$ is
strongly connected and has at least one self-link for all
$n\in\mathbb N$ and $p\in\underline{P}$ almost surely, where $\hat{B}_{p}$ is defined in the proof of Theorem \ref{thm3}, for example, (\ref{reorder}) in Sec.~5.3. Then the
algorithm (\ref{delayed}) synchronizes to a $P$-periodic trajectory.
In particular if $P=1$, then the algorithm (\ref{delayed}) reaches
consensus.
\end{proposition}

\section{Applications}

Adapted processes are rather general and include i.i.d processes and
Markov chains as two special cases. Therefore, the results obtained
above can be directly utilized to derive sufficient conditions for
the cases where the topology switching and delays are i.i.d. or
Markovian.

First, by a standard construction as mentioned in Sec.~1, from the property of i.i.d. it follows
that $\mathbb E\{G(\sigma^{k+1})|\mathcal{F}^{k}\}=\mathbb E\{G(\sigma^{k+1})\}$ is
a constant stochastic matrix. Then, we have the following results.
\begin{corollary}
Assume that $\mathbf A$ holds and $\{\sigma^{t}\}$ is an i.i.d.
process. Suppose there exist $\mu>0$, $L\in\mathbb N$, and $\delta>0$
such that $G^{0}(\sigma)>\mu I_{m}$ for all $\sigma\in\Omega$ and
the $\delta$-graph of $ \mathbb E\{G(\sigma^{1})\}$ has a spanning
tree. Then the delayed multi-agent system via algorithm
(\ref{delayed}) reaches consensus.
\end{corollary}

\begin{corollary}
Assume that $\mathbf A$ and $\mathbf B$ hold and $\{\sigma^{t}\}$
is an i.i.d. process. Suppose there exist $L\in\mathbb N$ and $\delta>0$
such that the $\delta$-graph of $\mathbb
E\{\hat{G}^{0}(\sigma^{1})\}$ is strongly connected for all
$n\in\mathbb N$ almost surely. Then the system (\ref{delayed})
synchronizes to a $P$-periodic trajectory. In particular, if $P=1$,
then (\ref{delayed}) reaches consensus.
\end{corollary}

Second, we consider the Markovian switching topologies, namely, the
graph sequence is induced by a homogeneous Markov chain with a
stationary distribution and the property of uniform ergodicity,
which is defined as follows.
\begin{definition}\cite{Chilina}
    A Markov chain $\{\sigma^{t}\}$, defined on $\{\Omega,\mathcal F\}$, with a stationary distribution $\pi$ and a
    transition probability $\mathbb{T}(x,A)$ is called uniformly
    ergodic if
    \begin{eqnarray*}
      \sum_{x\in\Omega}\|\mathbb{T}^{k}(x,\cdot)-\pi(\cdot)\|\to
      0 \textnormal{~as~}k\to+\infty,
    \end{eqnarray*}
where $\mathbb T^{k}(\cdot,\cdot)$ denotes the $k$-th iteration of the transition probability
$\mathbb T(\cdot,\cdot)$, for two probability measures $\mu$ and $\nu$ on
$\{\Omega,\mathcal F)\}$, and
$\|\mu-\nu\|=\sup_{\mathcal A\in\mathcal F}|\mu(\mathcal A)-\nu(\mathcal A)|$.
\end{definition}
From the Markovian property, we have the following results.

\begin{corollary}\label{Markov1}
Assume that $\mathbf A$ holds. Let $\{\sigma^{t}\}$ be an irreducible and
aperiodic Markov chain with a unique invariant measure $\pi$. Suppose $\{\sigma^{t}\}$ is uniformly ergodic and there exist
$\mu>0$ and $\delta>0$ such that
$G^{0}(\sigma)>\mu I_{m}$ for all $\sigma\in\Omega$ and the
$\delta$-graph of $\mathbb E_{\pi}\{G(\sigma^{1})\}$ has a spanning
tree. Then the delayed multi-agent system(\ref{delayed}) reaches consensus.
\end{corollary}

\begin{proof}
From the Markovian property, we have
\begin{eqnarray*}
E\{\frac{1}{L}\sum_{t=n+1}^{n+L}G(\sigma^{t})|\mathcal{F}^{n}\}
=E\{\frac{1}{L}\sum_{t=n+1}^{n+L}G(\sigma^{t})|\sigma^{n}\}.
\end{eqnarray*}
If $\{\sigma^{t}\}$ is uniformly ergodic, then
\begin{eqnarray*}
\lim_{L\to +\infty}\mathbb E\{\frac{1}{L}\sum_{1=n+1}^{n+L}G(\sigma^{t})|\sigma^{n}\}
=\lim_{L\to +\infty}\frac{1}{L}\sum_{i=1}^{L}\int_{\Omega}
G(y)\mathbb{T}^{i}(\sigma^{n},dy)
=\int_{\Omega}G(y)\pi(dy)=\mathbb E_{\pi}[G(\sigma^{1})].
\end{eqnarray*}
Since the convergence is uniform, there exits $L$ such that
the $\delta/2$-graph corresponding to
$\mathbb E\{(1/L)\sum_{t=n+1}^{n+L}G(\sigma^{t})
|\mathcal{F}^{n}\}$ has a spanning tree almost surely.
From
Theorem \ref{thm2}, the conclusion can be derived.
\end{proof}

\begin{corollary}\label{Markov2}
Assume that $\mathbf A$ and $\mathbf B$ hold, and let $\{\sigma^{t}\}$ be
an irreducible and aperiodic Markov chain with a unique invariant
measure $\pi$. Suppose that  $\{\sigma^{t}\}$ is uniformly ergodic and there exists $\delta>0$ such that the
$\delta$-graph of $\mathbb E_{\pi}\{\hat{G}^{0}(\sigma^{1})\}$ is
strongly connected. Then the system (\ref{delayed}) synchronizes to
a $P$-periodic trajectory. In particular, if $P=1$, then
(\ref{delayed}) reaches consensus.
\end{corollary}

These corollaries can be proved directly from Theorems \ref{thm3} in the same way as Corollary \ref{Markov1}.
It can be seen that the a homogeneous Markov chain with finite state space and unique invariant distribution is uniformly ergodic. Hence, the results of Corollaries \ref{Markov1} and \ref{Markov2} hold for this scenario.

\section{Proofs of the main results}

In the following, the coupling matrix $B(\cdot)$ in
the delayed system (\ref{integ}) is written in the following block
form:
\begin{eqnarray*}
B(\sigma^{t})
=\left[\begin{array}{cccc}B_{1,1}(\sigma^{t})&B_{1,2}(\sigma^{t})&\cdots
&B_{1,\tau_{M}+1}(\sigma^{t})\\B_{2,1}(\sigma^{t})&B_{2,2}(\sigma^{t})&\cdots
&B_{2,\tau_{M}+1}(\sigma^{t})\\
\vdots&\vdots&\ddots&\vdots\\
B_{\tau_{M}+1,1}(\sigma^{t})&B_{\tau_{M}+1,2}(\sigma^{t})&\cdots
&B_{\tau_{M}+1,\tau_{M}+1}(\sigma^{t})\end{array}\right]\in\mathbb
R^{(\tau_{M}+1)m,(\tau_{M}+1)m}\label{B}
\end{eqnarray*}
with $B_{ij}(\sigma^{t})\in\mathbb R^{m,m}$,
$i,j\in\underline{\tau_{M}+1}$. For two index sets $I$ and $J$, we
denote by $[B(\sigma^{t})]_{I,J}$ the sub-matrix of $B(\sigma^{t})$
with row index set $I$ and column index set $J$. For an $n$-length
word $\sigma=(\sigma^{k})_{k=1}^{n}$ in the stochastic process, we
use $B(\sigma)$ to represent the matrix product
$\prod_{i=1}^{n}B(\sigma^{i})$. One can see that the structure of the matrix $B(\sigma^{t})$ has the following properties: (1). Each $B_{i,i-1}=I_{m}$ for all $i\ge 2$; (2). $B_{i,j}=0$ for all $i\ge 2$ and $j\ne i-1$. These properties are essential for the following proofs.

As the same way defined below Eq. (\ref{integ}), let us consider the corresponding graph $\mathcal G(B(\sigma^{t}))$,
which has $(\tau_{M}+1)m$ vertices, which we denote by
$\{v_{i,j},~i\in\underline{\tau_{M}+1},~j\in\underline{m}\}$, where
$v_{i,j}$ corresponds to the $(i-1)m+j$ row of the matrix
$B(\sigma)$.

We denote the following finitely generated periodic group:
$$\langle i_{1},i_{2},\dots,i_{K} \rangle_{j} := \{p:~p=\sum_{l=k}^{K}i_{k}p_{k}~{\rm
mod}~j,~p_{k}\in\mathbb Z\}.$$
If these numbers are
be picked in a finite integer set, for instance,
$\{1,\dots,\tau_{M}+1\}$ in the present paper, then $\langle
i_{1},i_{2},\dots,i_{K} \rangle_{j}$ denotes the set $\langle
i_{1},i_{2},\dots,i_{K} \rangle_{j}\bigcap\underline{\tau_{M}+1}$
unless specified otherwise. As a default setup, $\langle i\rangle$
denotes $\langle i\rangle_{\tau_{0}+1}$ where $\tau_{0}$ is the
self-linking delay as in  (\ref{delayed_matrix}). We will sometimes
be interested in whether an element in a matrix is zero or not,
regardless of its actual value.

\subsection{Proof of Theorem \ref{basic_thm}}
From the condition in this theorem, we can see that the $\delta$-matrix of $\mathbb E\{\prod_{k=n+1}^{n+L}G(\sigma^{k})|\mathcal
F^{n}\}$ is SIA for all $n\in\mathbb N$. Lemma \ref{SIA} states that there exists $N\in\mathbb
N$ such that the product of any $N$ SIA matrices in $\mathbb
R^{m,m}$ is scrambling. Note that
\begin{eqnarray*}
&&\mathbb E\bigg\{\prod_{t=n+1}^{n+NL}G(\sigma^{t})|\mathcal
F^{n}\bigg\}=\mathbb E\bigg\{\cdots\mathbb E\bigg\{\mathbb E
\bigg\{\prod_{t_{L}=n+(N-1)L+1}^{n+NL}G(\sigma^{t_{L}})|\mathcal
F^{n+(N-1)L}\bigg\}\\
&&\prod_{t_{L-1}=n+(N-2)L+1}^{n+(N-1)L}G(\sigma^{t_{L-1}})| \mathcal
F^{n+(N-2)L}\bigg\}\cdots\prod_{t_{1}=n+1}^{n+L}G(\sigma^{t_{1}})|\mathcal
F^{n}\bigg\},
\end{eqnarray*}
since $\{\mathcal F^{t}\}$ is a filtration.  This implies that there exists a positive constant
$\delta_{1}<\delta^{N}$ such that the $\delta_{1}$-graph of $\mathbb
E\{\prod_{t=n+1}^{n+NL}G(\sigma^{t})|\mathcal F^{n}\}$ is
scrambling. So, from Lemma 3.12 in Ref. \cite{LLC}, there exist $\delta'>0$
and $M_{1}\in\mathbb N$ such that
\begin{eqnarray*}
\mathbb
P\bigg\{\eta\bigg(\prod_{t=n+1}^{n+M_{1}NL}G(\sigma^{t})\bigg)>\delta'|\mathcal
F^{n}\bigg\}>\delta',~\forall~n\in\mathbb N.
\end{eqnarray*}
Let $C_{k}=\prod_{t=kM_{1}NL+1}^{(k+1)M_{1}NL}G(\sigma^{t})$. We can
conclude that for almost every sequence of $\{\sigma^{t}\}$, it
holds that
\begin{eqnarray*}
\lim_{K\to\infty}\sum_{k=1}^{K}\mathbb
P\bigg\{\eta(C_{k})>\delta'|\mathcal
F^{kNL}\bigg\}>\lim_{K\to\infty}K\times\delta'=+\infty.
\end{eqnarray*}
From Lemma \ref{lemSecondBorelCantelli}, we can conclude that the
events $\{\eta(C_{k})>\delta'\}$, $k=1,2,\dots,$ occur infinitely often
almost surely. Therefore, we can complete the proof directly from
Lemma \ref{basic_lemma}.

\subsection{Proof of Theorem \ref{thm2}}

The proof of this theorem is based on the structural characteristics
of the product of matrices $B(\cdot)$. We denote by
$[B(\cdot)]_{ij}$ the $\mathbb R^{m,m}$ sub-matrix of $B(\sigma)$ in
the position $(i,j)$. We first show by the following lemma that the graph corresponding to the product of more than $\tau_{M}+1$ successive matrices $B(\sigma^{t})$, as defined by (\ref{integ}), has a spanning tree and self-link at one root vertex. Thus, we can prove Theorem \ref{thm2} by employing Theorem \ref{basic_thm}.

\begin{lemma}\label{lem_con_matrix}
Under the conditions in Theorem \ref{thm2}, for any $n$-length word
$\sigma=(\sigma_{i})_{i=1}^{n}$  with $n\ge\tau_{M}+1$, there exists $\mu_{1}>0$ such that
\begin{enumerate}[(i).]
\item $[B(\sigma)]_{i,1}\ge \mu_{1}^{n}I_{m}$;
\item $\sum_{j=1}^{\tau_{M}+1}[B(\sigma)]_{1,j}\ge\mu_{1}^{n}
\sum_{j=1}^{\tau_{M}+1}\sum_{k=1}^{n}G^{j}(\sigma^{k})$.
\end{enumerate}
\end{lemma}

\bigskip
\begin{proof} We choose $0<\mu_{1}<\mu$, where $\mu$ is defined in Theorem \ref{thm2}. (i). For a word $\sigma=(\sigma_{i})_{i=1}^{n}$ with
$n\ge\tau_{M}+1$,
\begin{eqnarray*}
[B(\sigma)]_{i,1}&=&\sum_{i_{1},\dots,i_{n}}[B(\sigma_{n})]_{i,i_{1}}
[B(\sigma_{n-1})]_{i_{1},i_{2}}\cdots
[B(\sigma_{1})]_{i_{n,1}}\\
&\ge&\bigg(\prod_{k=n-i+2}^{n}[B(\sigma_{k})]_{k+i-n,k+i-n-1}\bigg)
\bigg(\prod_{k=1}^{n-i+1}[B(\sigma_{k})]_{1,1}\bigg)\\
&=&\prod_{k=1}^{n-i+1}[B(\sigma_{k})]_{1,1}
\ge\mu^{n}_{1}I_{m}
\end{eqnarray*}
since $[B(\varpi)]_{k+i-n,k+i-n-1}=I_{m}$ for all $k\ge n-i+2$ and $[B(\varpi)]_{1,1}\ge
\mu I_{m}\ge\mu_{1}I_{m}$ for all $\varpi\in\Omega$.

(ii). Let $j\in\underline{\tau_{M}+1}$ and
$t_{0}\in\underline{n}$. If $t_{0}\ge j$, we have
\begin{eqnarray*}
\sum_{l}[B(\sigma)]_{1,l}&=&\sum_{i_{1},\dots,i_{n},l}[B(\sigma_{n})]
_{1,i_{1}}[B(\sigma_{n-1})]_{i_{1},i_{2}}
\cdots[B(\sigma_{1})]_{i_{n},l}\\
&\ge&\bigg(\prod_{k=t_{0}+1}^{n}[B(\sigma_{k})]_{1,1}\bigg)[B(\sigma_{t_{0}})]
_{1,j}
\bigg(\prod_{l=t_{0}-j+2}^{t_{0}-1}[B(\sigma_{l})]_{l-t_{0}+j,l-t_{0}+j-1}\bigg)\\
&&\bigg(\prod_{p=1}^{t_{0}-j+1}[B(\sigma_{p})]_{1,1}\bigg)\\
&\ge&\mu^{n}_{1}[B(\sigma_{t_{0}})]_{1,j},
\end{eqnarray*}
since $[B(\varpi)]_{1,1}\ge\mu_{1} I_{m}$, $[B(\varpi)]_{l-t_{0}+j,l-t_{0}+j-1}=I_{m}$ for all $l\ge t_{0}-j+2$ for all $\varpi\in\Omega$;
whereas if $j>t_{0}$, we similarly have
\begin{eqnarray*}
\sum_{l}[B(\sigma)]_{1,l}&\ge&\bigg(\prod_{k=t_{0}+1}^{n}[B(\sigma_{n})]
_{1,1}\bigg)
[B(\sigma_{t_{0}})]_{1,j}\bigg(\prod_{l=1}^{t_{0}-1}
[B(\sigma_{l})]_{l+j-t_{0}+1,l+j-t_{0}}\bigg)\\
&\ge&\mu^{n}_{1}[B(\sigma_{t_{0}})]_{1,j}.
\end{eqnarray*}
Summing the right-hand side of the above inequality with
respect to $t_{0}$ and $j$proves (ii).
\end{proof}

{\it Proof of Theorem \ref{thm2}.} Let us consider the $\mu_{1}^{n}$-graph of
$B(\sigma)$ for all $\sigma=(\sigma^{t})_{t=1}^{n}$ with $n\ge\tau_{M}+1$, as defined in Lemma \ref{lem_con_matrix}. The item (i) in
Lemma \ref{lem_con_matrix} indicates that for each vertex $v_{i,j}$
with $i\ge 2$ and $j\in\underline{m}$, there exist a path
from vertex $v_{1,j}$ to $v_{i,j}$:
$(v_{1,j},v_{2,j},\dots,v_{i,j})$.

From item (ii) in Lemma \ref{lem_con_matrix} and the conditions in
Theorem \ref{thm2}, one can see that there exits $\delta>0$ and
$L\in\mathbb N$ such that the $\delta$-graph of $\sum_{l}[\mathbb
E\{\prod_{t=n+1}^{n+L}B(\sigma^{t})|\mathcal F^{n}\}]_{1,l}$ has
spanning trees and self-links. Let $\mathcal G$ be the random
variable corresponding to the $\delta$-graph of $\mathbb
E\{\prod_{t=n+1}^{n+L}B(\sigma^{t})|\mathcal F^{n}\}$ and $\mathcal
G'$ be the random variable corresponding to the $\delta$-graph of
\allowbreak
$\sum_{l}[\mathbb E\{\prod_{t=n+1}^{n+L}B(\sigma^{t})|\mathcal
F^{n}\}]_{1,l}$. Then, for almost every graph $\mathcal G'$, there
exists an index $j_{0}\in\underline{m}$ such that for any $j$, there
exists a path $(j_{0},j_{1},\dots,j_{K-1},j)$ to access $j$. This
implies that for almost every graph $\mathcal G$, there exists a
path from $v_{1,j_{0}}$ to $v_{1,j}$. Thus, $v_{1,j_{0}}$ can access
all vertices $v_{i,j}$, $i=1,\dots,\tau_{M}+1$, since $v_{1,j}$ can
access all $v_{i,j}$ for $\tau_{M}+1\ge i\ge 2$ by a directed link
and $v_{1,j_{0}}$ and has self-link, noting that $G^{0}(\cdot) $ has
positive diagonals. Therefore, for almost every graph $\mathcal G$,
it has a spanning tree and the vertex $v_{1,j_{0}}$ is one of the
roots. From Lemma \ref{spanning}, one can see that $\mathbb
E\{\prod_{t=n+1}^{n+L}B(\sigma^{t})|\mathcal F^{n}\}$ is SIA almost
surely. According to Theorem \ref{basic_thm}, the system
(\ref{multi_undelay}) reaches consensus. This proves the theorem.

\subsection{Proof of Theorem \ref{thm3}}

{\it Outline of the proof:} For a better
understanding of the proof, we first give the following sketch.
We start the proof
by defining a permutation matrix $Q\in\mathbb R^{\tau_{M}+1,\tau_{M}+1}$
corresponding to the permutation sequence from
$(1,2\dots,\tau_{M}+1)$ to $(\langle 1\rangle,\langle
2\rangle,\dots,\langle P\rangle)$. Then we show by the lemma that follows
that the matrix $B(\sigma^{t})$ can be transformed into
the following form:
\begin{eqnarray}
[Q\otimes I_{m}]B(\sigma^{t})[Q\otimes I_{m}]^{\top}
=\left[\begin{array}{llll}\hat{B}_{1}(\sigma^{t})&0&\cdots&0\\
0&\hat{B}_{2}(\sigma^{t})&\cdots&0\\
\vdots&\vdots&\ddots&\vdots\\
0&0&\cdots&\hat{B}_{P}(\sigma^{t})\end{array}\right],\label{reorder}
\end{eqnarray}
where $\otimes$ stands for the Kronecker product and $\hat{B}_{p}(\sigma^{t})=B_{\langle \langle
p\rangle|_{P}\rangle,\langle \langle
p\rangle|_{P}\rangle}(\sigma^{t})$. By the permutation $Q$, we can
rewrite the coupled system (\ref{multi_delay1}) as
\begin{eqnarray}
\hat{y}^{t+1}=\hat{B}(\sigma^{t})\hat{y}^{t},\label{multi_delay}
\end{eqnarray}
where $\hat{y}^{t}=[Q\otimes I_{m}]y^{t}$ and
$\hat{B}(\cdot)=[Q\otimes I_{m}]B(\cdot)[Q\otimes I_{m}]^{\top}$.
This system can be divided into $P$ subsystem as
\begin{eqnarray}
\hat{y}_{p}^{t+1}=\hat{B}_{p}(\sigma^{t})\hat{y}_{p}^{t},~p\in\underline{P},
\label{equivalent}
\end{eqnarray}
where $\hat{y}_{p}^{t}$ corresponds to $[y^{t}]_{\langle\langle
p\rangle|P\rangle}$. So, it is sufficient to prove the following
claim to complete this proof from Lemma \ref{basic_thm}:

{\em Claim 1:} For each $p\in\underline{P}$, there exists
$\delta'>0$ and $L\in\mathbb N$ such that the $\delta'$-graph of the
matrix
\begin{eqnarray}
\mathbb E\bigg\{\prod_{t=n+1}^{n+L}\hat{B}_{p}(\sigma^{t})|\mathcal
F^{n}\bigg\}
\end{eqnarray}
has a spanning tree for all $n\in\mathbb N$ almost surely.

The proof of this theorem is also based on the structural
characteristics of the product of matrices $B(\cdot)$. By the lemmas below, we are to show the permutation form (\ref{reorder}) can be guaranteed.

\begin{lemma}\label{lem_syn_m1} Under the conditions of Theorem \ref{thm3}, for any $(\tau_{0}+1)$-length word
$\sigma=(\sigma_{k})_{k=1}^{\tau_{0}+1}$, there exists some $\mu_{1}>0$ such that the following hold:
\begin{enumerate}[(i).]
\item $[B(\sigma)]_{i,i}\ge\mu_{1}^{\tau_{0}+1}I_{m}$ for all
$i\in\underline{\tau_{0}+1}$;

\item
$[B(\sigma)]_{j,j-(\tau_{0}+1)}\ge I_{m}$ for all
$j\ge\tau_{0}+2$;

\item $\sum_{l\in\langle
1-j\rangle}[B(\sigma)]_{\tau_{0}+2-j,l}\ge
\mu^{\tau_{0}+1}_{1}\hat{G}^{0}(\sigma_{j})$ for all
$j\in\underline{\tau_{0}+1}$;

\item  $\sum_{l\in\langle
i+(\tau+1)\rangle}[B(\sigma)]_{i,l}\ge\mu_{1}^{\tau_{0}+1}
[B(\sigma_{\tau_{0}+2-i})]_{1,\tau+1}$
for all $i\in\underline{\tau_{0}+1}$ and
$\tau\in\underline{\tau_{M}}$.
\end{enumerate}
\end{lemma}

\bigskip
\begin{proof} We choose $0<\mu_{1}<\mu$. (i). For any $i\in\underline{\tau_{0}+1}$, we have
\begin{eqnarray*}
&&[B(\sigma)]_{i,i}=\sum_{i_{1},\dots,i_{\tau_{0}}}
[B(\sigma_{\tau_{0}+1})]_{i,i_{1}}
[B(\sigma_{\tau_{0}})]_{i_{1},i_{2}}\cdots[B(\sigma_{1})]_{i_{\tau_{0}},i}\\
&&\ge\bigg(\prod_{p=\tau_{0}+3-i}^{\tau_{0}+1}
[B(\sigma_{p})]_{p+i-1-\tau_{0},p+i-2-\tau_{0}}\bigg)
[B(\sigma_{\tau_{0}-i+2})]_{1,\tau_{0}+1}\\
&&\bigg(\prod_{q=1}^{\tau_{0}-i+1}
[B(\sigma_{q})]_{q+i,
q+i-1}\bigg)\ge\mu I_{m}\ge\mu_{1}^{\tau_{0}+1}I_{m}
\end{eqnarray*}
since $[B(\varpi)]_{i+1,i}=I_{m}$ and
$[B(\varpi)]_{1,\tau_{0}+1}\ge\mu I_{m}$ for all $\varpi\in\Omega$ and
$i\in\underline{\tau_{M}}$.

(ii). For any $j\ge\tau_{0}+2$, we have
\begin{eqnarray*}
&&[B(\sigma)]_{j,j-(\tau_{0}+1)}=\sum_{i_{1},\dots,\tau_{0}}
[B(\sigma_{\tau_{0}+1})]_{j,i_{1}}
[B(\sigma(\tau_{0}))]_{i_{1},i_{2}}\cdots[B(\sigma_{1})]
_{i_{\tau_{0}},j-(\tau_{0}+1)}\\
&&\ge\prod_{k=1}^{\tau_{0}+1}[B(\sigma_{k})]
_{k+j-\tau_{0}-1,k+j-\tau_{0}-2}=I_{m}
\end{eqnarray*}
since $[B(\varpi)]_{i+1,i}=I_{m}$ for all $i\ge 2$ and
$\varpi\in\Omega$.

(iii). For any $i\in\underline{\tau_{0}+1}$, we have
\begin{eqnarray*}
&&\sum_{i_{1},\dots,i_{\tau_{0}},k}[B(\sigma)]_{i,i+(\tau_{0}+1)k}
=\sum_{k}[B(\sigma_{\tau_{0}+1})]_{i,i_{1}}
[B(\sigma_{\tau_{0}})]_{i_{1},i_{2}}\cdots[B(\sigma_{1})]
_{i_{\tau_{0}},(\tau_{0}+1)k+i}\\
&&\ge\bigg(\prod_{k=2}^{i}[B(\sigma_{\tau_{0}-i+k+1})]_{k,k-1}\bigg)
[B(\sigma_{\tau_{0}-i+2})]_{1,(k+1)(\tau_{0}+1)}\\
&&\bigg(\prod_{l=i+(\tau_{0}+1)(k+1)-\tau_{0}}^{(k+1)(\tau_{0}+1)}
[B(\sigma_{l-k(\tau_{0}+1)-i})]_{l,l-1}\bigg)\ge[B(\sigma_{\tau_{0}+2-i})]_{1,(k+1)(\tau_{0}+1)}
\end{eqnarray*}
for all $k\ge 0$. Summing the right-hand side with respect to $k$
 and letting $j=\tau_{0}+2-i$, we have
$\sum_{l\in\langle
1-j\rangle}[B(\sigma)]_{\tau_{0}+2-j,l}\ge\sum_{l\in\langle\tau_{0}+1\rangle}
[B(\sigma_{j})]_{1,l}$.

(iv). Let $j=\tau_{0}+2-i$. If $j\ge\tau$,
\begin{eqnarray*}
&&\sum_{k}[B(\sigma)]_{\tau_{0}+2-j,\tau_{0}+2-j+(\tau+1)+(\tau_{0}+1)k}\ge
\bigg(\prod_{p=j+1}^{\tau_{0}+1}[B(\sigma_{p})]_{p-j+1,p-j}\bigg)\\
&&[B(\sigma_{j})]_{1,\tau+1}\bigg(\prod_{q=j-\tau}^{j-1}[B(\sigma_{q})]
_{q+\tau+2-j,q+\tau+1-j}\bigg)
[B(\sigma_{j-\tau-1})]_{1,\tau_{0}+1}\\
&&\bigg(\prod_{l=1}^{j-\tau-2}[B(\sigma_{l})]_{l+\tau+\tau_{0}+3-j,
l+\tau+\tau_{0}+2-j}\bigg)
\ge\mu[B(\sigma_{j})]_{1,\tau+1};
\end{eqnarray*}
whereas if $j<\tau$,
\begin{eqnarray*}
&&\sum_{k}[B(\sigma)]_{\tau_{0}+2-j,\tau_{0}+2-j+(\tau+1)+(\tau_{0}+1)k}\ge
\bigg(\prod_{p=j+1}^{\tau_{0}+1}[B(\sigma_{p})]_{p-j+1,p-j}\bigg)
[B(\sigma_{j})]_{1,\tau+1}\\
&&\bigg(\prod_{q=1}^{j-1}[B(\sigma_{q})]_{q+\tau+2-j,q+\tau+1-j}\bigg)\ge
[B(\sigma_{j})]_{1,\tau+1}.
\end{eqnarray*}
These calculations complete the proof of the lemma.
\end{proof}


\begin{lemma}\label{lem_syn_m2} Under the conditions of Theorem \ref{thm3},
consider an $L(\tau_{0}+1)$-length word
$\tilde{\sigma}=(\tilde{\sigma}_{1},\dots,\tilde{\sigma}_{L})$,
where each $\tilde{\sigma}_{l}=(\sigma_{l,i})_{i=1}^{\tau_{0}+1}$ is
a $(\tau_{0}+1)$-length word. If $L\ge\tau_{M}+1$, then there exists $\mu_{1}>0$ such that
\begin{enumerate}[(i).]
\item $[B(\tilde{\sigma})]_{j,i}\ge\mu_{1}^{(\tau_{0}+1)L}I_{m}$ for all
$j\in\langle i\rangle$ and $i\in\underline{\tau_{0}+1}$;

\item
$\sum_{l\in\langle
i\rangle}[B(\tilde{\sigma})]_{\tau_{0}+2-j,l}\ge\mu_{1}^{(\tau_{0}+1)L}
\sum_{k}\hat{G}^{0}(\sigma_{k,j})$
for all $j\in\underline{\tau_{0}+1}$;

\item $\sum_{j\in\langle
i+\tau+1\rangle}[B(\tilde{\sigma})]_{i,j}
\ge\mu_{1}^{\tau_{0}+1}\sum_{l\in\langle\tau+1\rangle}[B(\tilde{\sigma}
_{\tau_{0}+2-i})]_{1,l}$
for all $i\in\underline{\tau_{0}+1}$ and
$\tau\in\underline{\tau_{M}}$;

\item If $\tau'$ is such that
$\tau'+1\notin\langle\tau_{0}+1,\tau_{1}+1,\dots,\tau_{K}+1\rangle$
and $[B(\sigma_{1})]_{1,\langle\tau'+1\rangle}=0$ for all
$\sigma_{1}\in\Omega$, then $[B(\tilde{\sigma})]_{i,\langle
i+\tau'+1\rangle}=0$ for all $i\ge 1$.
\end{enumerate}
\end{lemma}

\begin{proof} We pick some $\mu_{1}<\mu$. (i). For $j\le\tau_{0}+1$, the proof is similar to the
proof of item (i) of Lemma \ref{lem_syn_m1}. For $j\ge\tau_{0}+2$,
we have
\begin{eqnarray*}
[B(\tilde{\sigma})]_{j,i}\ge\bigg(\prod_{l=l_{1}}^{L}
[B(\tilde{\sigma}_{l})]_{j-(L-l)(\tau_{0}+1),j-(L-l+1)(\tau_{0}+1)}\bigg)
\bigg(\prod_{p=1}^{l_{1}}[B(\tilde{\sigma}_{p})]_{i,i}\bigg)
\ge\mu_{1}^{(\tau_{0}+1)L}I_{m},
\end{eqnarray*}
where $l_{1}=L+1-(j-i)/(\tau_{0}+1)$ is an integer (noting
$j\in\langle i\rangle$), since\\
$[B(\tilde{\sigma}_{l})]_{j-(L-l)(\tau_{0}+1),j-(L-l+1)(\tau_{0}+1)}\ge I_{m}$ holds here, as mentioned in Lemma \ref{lem_syn_m1} (ii).

The items (ii) and (iii) can be proved by similar arguments as in
the proof of items (iii) and (iv) of Lemma \ref{lem_syn_m1}. It
remains to prove item (iv). In the following, we will prove a
slightly more general result,  namely that $[B(\sigma)]_{i,\langle
i+\tau'+1\rangle}=0$ for all words $\sigma$ having length
$L\in\langle \tau_{0}+1 \rangle$. Let
$\sigma=(\sigma_{i})_{i=1}^{L}$ be an arbitrary $L$-length word. We
calculate $[B(\sigma)]_{i,j}$ with $j\in\langle i+\tau'+1\rangle$ as
a sum of several matrix product terms:
\begin{eqnarray*}
[B(\sigma)]_{i,j}=\sum_{i_{1},\dots,i_{L-1}}[B(\sigma_{L})]_{i,i_{1}}
[B(\sigma_{L-1})]_{i_{1},i_{2}}\cdots
[B(\sigma_{1})]_{i_{L-1},j}.
\end{eqnarray*}
Since any zero factor yields zero product, we avoid zero factors in
the calculations. That is,  in the expression above, only factors of
the form $[B(\sigma_{l})]_{i+1,i}$ and $[B(\sigma_{l})]_{1,j}$ can
occur where $j\in\langle i+\tau'+1\rangle$ and
$\tau'+1\notin\langle\tau_{0}+1,\tau_{1}+1,\dots,\tau_{K}+1\rangle$.
Thus, letting $j_{1}=i$, we have
\begin{eqnarray*}
&&[B(\sigma)]_{i,j}=\sum_{j_{1},\dots,j_{V},V}\bigg\{\bigg
[\prod_{l=1}^{V}\bigg(\prod_{k_{l}
=L-\sum_{p=1}^{l}j_{p}+2}^{L-\sum_{p=1}^{l-1}j_{p}}
[B(\sigma_{k_{l}})]_{\sum_{p=1}^{l}j_{p}+k_{l}-L,
\sum_{p=1}^{l}j_{p}+k_{l}-L-1}\bigg)\\
&&[B(\sigma_{L-\sum_{p=1}^{l}j_{p}+1})]
_{1,j_{p}+1}\bigg]\bigg(\prod_{k_{V+1}=1}^{L-\sum_{p=1}^{V}j_{p}}
[B(\sigma_{k_{V+1}})]_{L-\sum_{p=1}^{V}j_{p}+k_{V+1},
\sum_{p=1}^{V}L-\sum_{p=1}^{V}j_{p}+k_{V+1}-1}\bigg)\bigg\},
\end{eqnarray*}
where each
$j_{p}\in\langle\tau_{0}+1,\tau_{1}+1,\dots,\tau_{K}+1\rangle$.
Suppose that the matrix product is nonzero. Then
$j=\sum_{p=1}^{V}j_{p}-L$, i.e.,
$\langle(i+\tau'+1)-(\sum_{p=1}^{V}j_{p}-L)\rangle=0$, which
implies $\langle \tau'+1-\sum_{p=2}^{V}j_{p}+L \rangle=0$. This
means that
$\tau'+1\in\langle\tau_{0}+1,\tau_{k}+1:~k=1,\dots,K\rangle$,
which contradicts the condition
$\tau'+1\notin\langle\tau_{0}+1,\tau_{k}+1:~k=1,\dots,K\rangle$.
The lemma is proved.
\end{proof}

{\it Proof of Theorem \ref{thm3}.} Consider the graph $\hat{\mathcal
G}^{\delta}(\sigma^{t})=\{\hat{\mathcal V},\hat{\mathcal
E}(\sigma^{t})\}$ on $(\tau_{M}+1)m$ vertices corresponding to the
$\delta$-graph of the matrix $B(\sigma^{t})$ as defined at the
beginning of this section. For  $L\in\mathbb N$ as fixed
in the main condition of Theorem \ref{thm3} and an arbitrary
fixed $m\in\mathbb N$, let ${\mathbb  B}=\mathbb
E\{\prod_{t=n+1}^{n+L}B(\sigma^{t})|\mathcal F^{n}\}$ and
$\hat{\mathcal G}^{\delta}$ be the random variable picked in the
$\delta$-graphs of $\mathbb B$.

First, we divide the graph $\hat{\mathcal G}^{\delta}$ into
$\tau_{M}+1$ subgraphs: $\mathcal G_{k}^{\delta}=\{\mathcal
V_{k},\mathcal E_{k}(\sigma^{t})\}$, $k\in\underline{\tau_{M}+1}$,
where $\mathcal V_{k}=\{v_{k,i}:~i\in\underline{m}\}$ corresponds to
the rows or columns of ${\mathbb  B}_{k,k}$ and the vertex $v_{k,i}$
corresponds the $i$-th row or column of the matrix ${\mathbb  B}_{k,k}$.
Then, integrate the subgraphs $\{\mathcal
G_{k}^{\delta}\}_{k=1}^{\tau_{M}+1}$ into $\tau_{0}+1$ subgraphs:
$\mathcal {G'}_{l}^{\delta}=\{\mathcal V'_{l},\mathcal E'_{l}\}$,
$l\in\underline{\tau_{0}+1}$, where $\mathcal
V_{l}'=\bigcup_{k\in\langle l \rangle}\mathcal V_{k}$,
$l\in\underline{\tau_{0}+1}$ and $\mathcal E_{l}'$ corresponds to
the intra-links in $\mathcal V_{l}'$. Let $\mathcal E_{l_{1},l_{2}}$
denote the inter-links from the subgraph of $\mathcal V_{l_{2}}'$ to
the subgraph $\mathcal V_{l_{1}}'$. Lemma \ref{lem_syn_m2} (i)
implies that for each $l\in\underline{\tau_{0}+1}$, there must exist
a link from $v_{l,i}$ to $v_{k,i}$ in the subgraph $\mathcal
{G'}_{l}^{\delta}(\cdot)$, for each vertex $v_{k,i}\in\mathcal
V_{k}$ with $k>l$ and $k\in\langle l\rangle$. Similarly to the the
proof of Theorem \ref{thm2}, the main condition of Theorem
\ref{thm3} and items (ii) and (iii) in Lemma \ref{lem_syn_m2} imply that
there exist $\delta_{1}>0$ and $L\in\mathbb N$ such that the
subgraph $\mathcal {G'}_{l}^{\delta_{1}}$ is strongly connected,
consequently having a spanning tree, and each vertex in $\mathcal
V_{l}$ is one of the roots in $\mathcal {G'}_{l}^{\delta_{1}}$ and
has a self-link almost surely for all $l\in\underline{\tau_{0}+1}$.

Second, according to
$\mbox{gcd}(\tau_{0}+1,\tau_{k}+1:~k\in\underline{K})=P$, we
integrate the subgraphs $\mathcal {G'}_{l}^{\delta_{1}}$ for all
$l\in\underline{\tau_{0}+1}$, into $P$ subgraphs, denoted by $\tilde{\mathcal
G}_{p}^{\delta_{1}}=\{\tilde{\mathcal V}_{p},\tilde{\mathcal
E}_{p}\}$, $p\in\underline{P}$ by $\tilde{\mathcal V}_{p}=\{\mathcal
V'_{j}:~\mathcal E_{j,p}\ne\emptyset\}$. The items (ii) and (iii) in Lemma
\ref{lem_syn_m2} and the second item in condition $\bf B$ indicate
that the $\delta_{1}$-matrix of $\sum_{j\in\langle
\tau_{k}+1:~k=0,1,\dots,K\rangle}{\mathbb  B}_{l,l+j}$ is positive for
all $l\in\underline{\tau_{0}+1}$. This implies that there exists at
least one link from $\mathcal {G'}_{l+j}^{\delta_{1}}$ to $\mathcal
{G'}_{l}^{\delta_{1}}$ and this link end in $\mathcal V_{l}$. So, in
the graph $\mathcal {G'}^{\delta_{1}}$, the root vertex in $\mathcal
{G'}_{l+j}^{\delta_{1}}$ can reach all vertices in $\mathcal
{G'}_{l}^{\delta_{1}}$ since each vertex in $\mathcal \mathcal
V_{l}$ is a root vertex in $\mathcal {G'}_{l}^{\delta_{1}}$. This
leads to the conclusion that $\mathcal V'_{j}\subset\tilde{\mathcal
V}_{l}$ provided $j-l\in\langle\tau_{k}+1:k=0,1,\dots,K\rangle$.
Also, we can conclude that each root vertex in ${\mathcal
{G'}_{l+j}}^{\delta_{1}}$ can reach all vertices in $\mathcal
{G'}_{l}^{\delta_{1}}$, by item (i) in Lemma \ref{lem_syn_m2}.
Therefore, we can conclude that $\tilde{\mathcal
V_{p}}=\bigcup_{l\in\langle p \rangle_{P}}\mathcal V'_{p}$ and each
$\tilde{\mathcal G}_{p}$ has a spanning tree almost surely. This
proves Claim 1. Moreover, there exists a vertex with self-link
in $\mathcal V_{i}$, $i\in\underline{\tau_{0}+1}$ and $i\in\langle
p\rangle_{P}$, as one of its roots, in $\hat{\mathcal G}^{\delta_{1}}$.
So, according to the arbitrariness of integer $n$, we can conclude that the
$\delta_{1}$-graph of $\mathbb
E\{\prod_{t=n+1}^{n+L}\hat{B}_{p}(\sigma^{t})|\mathcal F^{n}\}$ is
SIA almost surely for all $n\in\mathbb N$.

Finally, according to the second item in condition $\bf B$ and the
(iv) item in Lemma \ref{lem_syn_m2}, one can conclude that there are
no links between the graph $\tilde{\mathcal G}_{p}^{\delta}$ for
different $p\in\underline{P}$ for any $\delta\ge 0$. So, by a
permutation matrix $Q$ corresponding to the permutation sequence
from $(1,2\dots,\tau_{M}+1)$ to $(\langle 1\rangle,\langle
2\rangle,\dots,\langle P\rangle)$, $[Q\otimes
I_{m}]B(\sigma^{t})[Q\otimes I_{m}]^{\top}$ has the form
(\ref{reorder}).

By Theorem \ref{basic_thm}, we can conclude that (\ref{equivalent}) reaches
consensus for all $p=1,\dots,P$, but converges to different values
except for initial values in a set of Lebesgue measure zero.
Therefore, $x^{t}$ can synchronize and converge to a $P$-periodic
trajectory. This completes the proof of Theorem \ref{thm3}.

\section{Conclusions}

In this paper we have studied the convergence of the consensus
algorithm in multi-agent systems with stochastically switching
topologies and time delays. We have shown that consensus can be obtained
if the graph corresponding to the conditional expectations of the
coupling matrix product in consecutive times has spanning trees almost surely and
self-links are possible. With multiple delays, if self-links always
exist and are instantaneous (undelayed), then consensus can be guaranteed for
arbitrary bounded delays. Moreover, when the self-links are also delayed,
we have shown the phenomenon that the algorithm may not
reach consensus but instead may synchronize to a periodic trajectory
according to the delay patterns. Finally, we have  briefly
studied consensus algorithms without self-links. We have presented
several results for i.i.d. and Markovian switching topologies as
special cases.

\section*{Acknowledgments} W. Lu is supported by the National Key Basic Research and Development Program (No. 2010CB731403), the National Natural Sciences Foundation of China under Grant No. 60804044, the Foundation for the Author of National Excellent Doctoral Dissertation of China No. 200921, SGST 09DZ2272900, and LMNS, Fudan University.


\begin{thebibliography}{99}


\bibitem{Bli}
\newblock P.-A. Bliman and G. Ferrari-Trecate,
\newblock \emph{Average consensus
problems in networks of agents with delayed communications},
\newblock Automatica, \textbf{44} (2008), 1985--1995.



\bibitem{Cao1}
\newblock M. Cao, A. S. Morse and B. D. O. Anderson,
\newblock \emph{Reaching a
consensus in a dynamically changing environment: a graphical
approach},
\newblock SIAM J. Control Optim., \textbf{47} (2008), 575--600.

\bibitem{CS}
\newblock S. Chatterjee and E. Seneta,
\newblock \emph{Towards consensus: Some
convergence theorems on repeated averaging},
\newblock J. Appl. Prob., \textbf{14} (1977), 89--97.


\bibitem{Chilina}
\newblock O. Chilina,
\newblock ``f-uniform ergodicity of Markov chains'',
\newblock Supervised Project, Unversity of Toronto, 2006.

\bibitem{DeG}
\newblock M. H. DeGroot,
\newblock \emph{Reaching a consensus},
\newblock  J. Amer. Statist. Assoc., \textbf{69} (1974), 118--121.


\bibitem{DJ}
\newblock D. V. Dimarogonasa and K. H. Johansson,
\newblock \emph{Stability analysis
for multi-agent systems using the incidence matrix: Quantized
communication and formation control},
\newblock Automatica, \textbf{46} (2010), 695--700.

\bibitem{Durrett}
\newblock R. Durrett,
\newblock ``Probability: Theory and Examples'', 3$^{rd}$
\newblock edition,
\newblock Belmont, CA: Duxbury Press, 2005.

\bibitem{Fag}
\newblock F. Fagnani and S. Zampieri,
\newblock \emph{Average consensus with packet
drop communication},
\newblock SIAM J. Control Optim., \textbf{48} (2009), 102--133.



\bibitem{FAT}
\newblock L. Fang, P. J. Antsaklis and A. Tzimas,
\newblock \emph{Asynchronous
consensus protocols: Preliminary results, simulations and open
questions},
\newblock Proceedings of the 44th IEEE Conf. Decision and Control,
the Europ. Control Conference (2005), 2194--2199.



\bibitem{Fax}
\newblock J. A. Fax and R. M. Murray,
\newblock \emph{Information flow and cooperative
control of vehicle formations},
\newblock IEEE Trans. Autom. Control, \textbf{49} (2004), 1465--1476.

\bibitem{God}
\newblock C. Godsil and G. Royle,
\newblock ``Algebraic graph theory'',
\newblock Springer-Verlag, New York, 2001.

\bibitem{Haj1}
\newblock J. Hajnal,
\newblock \emph{The ergodic properties of non-homogeneous
finite Markov chains},
\newblock Proc. Camb. Phil. Soc., \textbf{52} (1956), 67--77.

\bibitem{Haj2}
\newblock J. Hajnal,
\newblock \emph{Weak ergodicity in non-homogeneous Markov
chains},
\newblock Proc. Camb. Phil. Soc., \textbf{54} (1958), 233--246.


\bibitem{Hat}
\newblock Y. Hatano and M. Mesbahi,
\newblock \emph{Agreement over random networks},
\newblock IEEE Trans. Autom. Control, \textbf{50} (2005), 1867--1872.

\bibitem{Horn}
\newblock R. A. Horn and C. R. Johnson,
\newblock ``Matrix analysis'',
\newblock Cambridge
University Press, 1985.




\bibitem{Kur}
\newblock Y. Kuramoto,
\newblock ``Chemical Oscillations, waves, and turbulence'',
\newblock Springer-Verlag, New York, 1984.


\bibitem{LMA}
\newblock J. Lin, A. S. Morse and B. D. O. Anderson,
\newblock \emph{The multi-agent
rendezvous problem Part 2: The asynchronous case},
\newblock SIAM J. Control Optim., \textbf{46} (2007), 2120--2147.

\bibitem{LLC}
\newblock B. Liu, W. Lu and T. Chen,
\newblock \emph{Consensus in networks of multiagents with switching topologies modeled as adapted stochastic
processes},
\newblock SIAM J. Control Optim., \textbf{49} (2011), 227--253.


\bibitem{Lu1}
\newblock W. Lu, F. M. Atay and J. Jost,
\newblock \emph{Synchronization of
discrete-time networks with time-varying couplings},
\newblock SIAM J. Math. Analys., \textbf{39} (2007), 1231--1259.

\bibitem{Lu2}
\newblock W. Lu, F. M. Atay and J. Jost,
\newblock \emph{Chaos synchronization in networks of coupled maps
with time-varying topologies},
\newblock Eur. Phys. J. B, \textbf{63} (2008), 399--406.



\bibitem{Lyn}
\newblock  N. A. Lynch,
\newblock ``Distributed algorithms'',
\newblock CA: Morgan Kaufmann, San Francisco, 1996.


\bibitem{Mes}
\newblock W. Ni and D. Z. Cheng,
\newblock \emph{Leader-following consensus of
multi-agent systems under fixed and switching topologies},
\newblock Systems \& Control Letters, \textbf{59} (2010), 209--217.


\bibitem{Mich}
\newblock W. Michiels, C.-I. Mor\u{a}rescu and S.-I. Niculescu,
\newblock \emph{Consensus problems with distributed delays, with application to
traffic flow models},
\newblock SIAM J. Control Optim., \textbf{48} (2009), 77--101.

\bibitem{Mor1}
\newblock L. Moreau,
\newblock \emph{Stability of continuous-time distributed
consensus algorithms},
\newblock 43rd IEEE Conference on Decision and Control, \textbf{4} (2004), 3998--4003.


\bibitem{Mor}
\newblock  L. Moreau,
\newblock \emph{Stability of multiagent systems with time-dependent communication links},
\newblock IEEE Trans. Autom. Control, \textbf{50} (2005), 169--182.


\bibitem{Olf}
\newblock R. Olfati-Saber and J. S. Shamma,
\newblock \emph{Consensus filters for
sensor networks and distributed sensor fusion},
\newblock 44th IEEE Conference
on Decision and Control 2005, and 2005 European Control Conference
CDC-ECC '05. 6698--6703.


\bibitem{Olf1}
\newblock R. Olfati-Saber, J. A. Fax and R. M. Murray,
\newblock \emph{Consensus and
cooperation in networked multi-agent systems},
\newblock Proceedings of the IEEE, \textbf{95} (2007), 215--233.


\bibitem{Olf2}
\newblock R. Olfati-Saber and R. M. Murray,
\newblock \emph{Consensus problems in
networks of agents with switching topology and time-delays},
\newblock IEEE Trans. Autom. Control, \textbf{49} (2004),1520--1533.

\bibitem{Pik}
\newblock A. Pikovsky, M. Rosenblum and J. Kurths,
\newblock \emph{Synchronization: A
universal concept in nonlinear sciences},
\newblock Cambridge University Press, 2001.


\bibitem{Shen}
\newblock J. Shen,
\newblock \emph{A geometric approach to ergodic non-homogeneous
Markov chains},
\newblock Wavelet Anal. Multi. Meth., LNPAM, \textbf{212} (2000), 341--366.



\bibitem{AA}
\newblock A. Tahbaz-Salehi and A. Jadbabaie,
\newblock \emph{A necessary and
sufficient condition for consensus over random networks},
\newblock IEEE Trans. Autom. Control, \textbf{53} (2008), 791--795.


\bibitem{Vic}
\newblock T. Vicsek, A. Czir\'{o}k, E. Ben-Jacob, I. Cohen and O. Shochet,
\newblock \emph{Novel type of phase transition in a system of self-driven
particles},
\newblock Phys. Rev. Lett, \textbf{75} (1995), 1226--1229.



\bibitem{Win}
\newblock A. T. Winfree,
\newblock ``The Geometry of biological time'',
\newblock Springer Verlag, New York, 1980.


\bibitem{Wolf}
\newblock J. Wolfowitz,
\newblock \emph{Products of indecomposable,
aperiodic, stochastic matrices},
\newblock Proceedings of AMS, \textbf{14} (1963), 733--737.


\bibitem{Wu}
\newblock C. W. Wu,
\newblock \emph{Synchronization and convergence of linear dynamics in random directed networks},
\newblock IEEE Trans. Autom. Control, \textbf{51}(2006), 1207--1210.

\bibitem{XW}
\newblock F. Xiao and L. Wang,
\newblock \emph{Consensus protocols
for discrete-time multi-agent systems with time-varying delays},
\newblock Automatica, \textbf{44} (2008), 2577--2582.

\bibitem{XW1}
\newblock F. Xiao, L. Wang,
\newblock \emph{Asynchronous consensus in
continuous-time multi-agent systems with switching topology and
time-varying delays},
\newblock IEEE Transactions on Automatic Control, \textbf{53} (2008), 1804--1816.

\bibitem{Zhang}
\newblock Y. Zhang and Y.-P. Tian,
\newblock \emph{Consentability and protocol design
of multi-agent systems with stochastic switching topology},
\newblock Automatica, \textbf{45} (2009), 1195--1201.




\end{thebibliography}
\end{document}